\documentclass{article}
\usepackage{amsthm}

\usepackage{newtxtext}       %
\usepackage{newtxmath}       

\usepackage[utf8]{inputenc}
\usepackage[T1]{fontenc}

\usepackage{algorithm}
\usepackage{algpseudocode}
\algnewcommand\algorithmicinput{\textbf{Input:}}
\algnewcommand\Input{\item[\algorithmicinput]}
\algnewcommand\algorithmicoutput{\textbf{Output:}}
\algnewcommand\Output{\item[\algorithmicoutput]}

\usepackage{hyperref}
\usepackage[english]{cleveref}
\usepackage{autonum}

\usepackage{dsfont}
\usepackage{nicefrac}
\usepackage{pgfplots}

\DeclareMathOperator{\vect}{vec}
\DeclareMathOperator{\diag}{diag}
\DeclareMathOperator{\spa}{span}
\DeclareMathOperator{\Real}{Re}
\DeclareMathOperator{\Imag}{Im}
\DeclareMathOperator{\rank}{rank}

\newtheorem{theorem}{Theorem}[section]
\newtheorem{lemma}[theorem]{Lemma}

\theoremstyle{definition}

\begin{document}
\title{A link between gramian based model order reduction and moment matching}
\author{C. Bertram\textsuperscript{1} and H. Faßbender\textsuperscript{1}}
\date{\textsuperscript{1}Institute for Numerical Analysis, Technische Universität Braunschweig, Germany}

\maketitle

\abstract{We analyze a family of Runge-Kutta based quadrature algorithms for the approximation of the gramians of linear time invariant dynamical systems. The approximated gramians are used to obtain an approximate balancing transformation similar to the approach used in balanced POD. It is shown that hereby rational interpolation is performed, as the approximants span certain Krylov subspaces. The expansion points are mainly determined by the time step sizes and the eigenvalues of the matrices given by the Butcher tableaus.}

\section{Introduction}
\label{sec:intro}
Consider a stable, minimal, linear time invariant single-input single-output continuous-time dynamical system
\begin{align}
\label{def:lti}
\begin{split}
	\dot{x}(t) &= Ax(t) + Bu(t), \quad x(0) = x_0,\\
	y(t) &= Cx(t),
\end{split}
\end{align}
with $A\in \mathbb{R}^{n\times n}$, $B\in \mathbb{R}^{n\times 1}$, $C\in \mathbb{R}^{1\times n}$ and therefore $x(t)\in \mathbb{R}^{n}$, $u(t) \in \mathbb{R}$ and $y(t)\in \mathbb{R}$. Stability of the system implies $\sigma(A)\subseteq \mathbb{C}_{-}$, i.e. all eigenvalues of $A$ have negative real part. The system matrix $A$ is assumed to be large and sparse.

The problem addressed here is to approximate the system \eqref{def:lti} by another system
\begin{align}
	\label{def:redlti}
\begin{split}
	\dot{\hat{x}}(t) &= \hat{A} \hat{x}(t) + \hat{B}u(t),\\
	\hat{y}(t) &= \hat{C}\hat{x}(t)
\end{split}
\end{align}
with possibly complex reduced system matrices $\hat{A}\in \mathbb{C}^{r\times r}$, $\hat{B}\in \mathbb{C}^{r\times 1}$, $\hat{C}\in \mathbb{C}^{1\times r}$ and $r \ll n.$ Among the many approaches for model order reduction (see, e.g., \cite{BenCOW2017} and the references therein) we will pursue a projection-based approach:  two  $n \times r$ matrices $V,W\in \mathbb{C}^{n\times r}$ with $W^\mathsf{H}V=I_r$ are  computed  which  define  the projector $\Pi =V W^\mathsf{H}.$ The  projection  of  the  states  of  the original system generates the reduced-order model with system matrices
\begin{align}\label{eq:proj}
 \hat{A} &= W^\mathsf{H} AV,\qquad
 \hat{B} =W^\mathsf{H} B,\qquad
 \hat{C} = CV.
\end{align}%
In practice, to obtain a real reduced system, the projection matrices can be kept real, but for ease of notation and explanation we allow for a complex valued projection.

In particular we focus on a balancing related approach which is derived from numerical integration with Runge-Kutta methods. We demonstrate how the reduced system generated by the method presented in this work can also be obtained by rational interpolation.
Thus the transfer functions of the original and the reduced order system coincide at certain interpolation points. We give an explicit formulation of those interpolation points in terms of the time step sizes used in the Runge-Kutta method and the eigenvalues of the matrix which determines the Butcher tableau representing the Runge-Kutta method.

\subsection{Balancing of LTI systems}\label{sec:bt}
Balancing is closely related to the controllability gramian $\mathcal{P}$ and the observability gramian $\mathcal{Q}$ of the system (\ref{def:lti}) \cite{Ant2005,Moo81,Lal02}. The gramians are defined as
\begin{align}
	\label{def:gramianP}
	\mathcal{P} = \int_0^{\infty}\! e^{At}B B^{\mathsf{T}}e^{A^{\mathsf{T}}t}\,\mathrm{d}t \in \mathbb{R}^{n \times n},\\
	\mathcal{Q} = \int_0^{\infty}\! e^{A^{\mathsf{T}}t}C^\mathsf{T} Ce^{At}\,\mathrm{d}t \in \mathbb{R}^{n \times n}.
\label{def:gramianQ}
\end{align}
The gramians $\mathcal{P}$ and $\mathcal{Q}$ are positive definite matrices as $A$ is stable. Thus, their Cholesky decompositions $\mathcal{P}=SS^{\mathsf{T}}$ and $\mathcal{Q}=RR^{\mathsf{T}}$ can be determined.
Let $U\Sigma T^\mathsf{T}$ be a singular value decomposition of $R^{\mathsf{T}}S$. Then $P=\Sigma^{-\frac{1}{2}}U^\mathsf{T}R^{\mathsf{T}}$ and $P^{-1}=ST^{-\mathsf{T}}\Sigma^{-\frac{1}{2}}$ define a balancing transformation.
That is, the gramians $\mathcal{\hat P}=P\mathcal{P}P^{\mathsf{T}}$ and $\mathcal{\hat Q}=P^{-\mathsf{T}}\mathcal{Q}P^{-1}$ of the transformed system
\begin{equation}
	\label{def:ltitrafo}
\begin{split}
	\dot{x}(t) &= PAP^{-1}x(t) + PBu(t), \quad x(0) = x_0,\\
	y(t) &= CP^{-1}x(t),
\end{split}
\end{equation}
are equal and diagonal.
Thus, in the balanced system it holds $\mathcal{\hat P}=\mathcal{\hat Q}=\Sigma$ with the Hankel singular values on the diagonal, which are an indicator for the importance of the corresponding state. The Hankel singular values are the square roots of the eigenvalues of the product of the gramians. They are invariant under state space transformations, i.e. the same for $\mathcal{P}\mathcal{Q}$ and $\hat{\mathcal{P}}\hat{\mathcal{Q}}=P\mathcal{P}\mathcal{Q}P^{-1}$ as they are similar.

In the model order reduction method balanced truncation a projection is performed onto the $r$ most important states, i.e. the states with large Hankel singular values. The projection $\Pi = VW^\mathsf{T}$ for the reduction process is derived from the partitioned singular value decomposition
\begin{align}\label{svd_trunc}
	R^{\mathsf{T}}S=\begin{bmatrix}U_r&U_0\end{bmatrix}\begin{bmatrix}\Sigma_r&\\&\Sigma_0\end{bmatrix}\begin{bmatrix}T_r^\mathsf{T}\\T_0^\mathsf{T}\end{bmatrix}
\end{align}
with $\Sigma_r \in \mathbb{R}^{r \times r}$, $U_r \in \mathbb{R}^{n \times r}$ and $T_r \in \mathbb{R}^{n \times r}$. The matrices $V$ and $W$ are obtained as
\begin{align}\label{eq:WV}
 W= R U_r \Sigma_r^{-\frac{1}{2}},\ V=ST_r\Sigma_r^{-\frac{1}{2}}
\end{align}
and indeed $W^{\mathsf{T}}V=I_r$ holds. For more details on the gramians and the energy associated with reaching/observing a state see, e.g., \cite[Chp. 4.3]{Ant2005}.

A bottleneck in this approach is the calculation of the gramians $\mathcal{P}$ and $\mathcal{Q}$ and their Cholesky factors. Different methods have been proposed for this situation, see, e.g., \cite{BenHTM11}, \cite{simoncini2016} and the references therein. A key idea to make calculations for large systems computationally feasible is to approximate the gramians with low-rank factors, i.e. $Z_{\text{c}}Z_{\text{c}}^{\mathsf{T}}\approx \mathcal{P}$ and $Z_{\text{o}}Z_{\text{o}}^{\mathsf{T}}\approx \mathcal{Q}$ with $Z_{\text{c}}\in \mathbb{R}^{n\times r_{\text{c}}}$, $Z_{\text{o}}\in \mathbb{R}^{n\times r_{\text{o}}}$ and $r_{\text{c}},r_{\text{o}}\ll n$. These approximate Cholesky factors $Z_\text{c}$ and $Z_\text{o}$ are then used instead of the actual Cholesky factors $S$ and $R$ to compute an approximate balancing transformation. This also includes a reduction of the system dimension as the number of columns in the approximate Cholesky factors is smaller than $n$.

In the method balanced proper orthogonal decomposition (BPOD) of snapshots as discussed in \cite{Rowley05}, but without output projection (see \cite{willcox2002} for a related approach) the gramians are approximated with finite sums. In particular the controllability gramian is approximated via
	\begin{align}
		\mathcal{P}=\int_0^{\infty}\! h(t)h(t)^\mathsf{T} \,\mathrm{d}t 
		&\approx \int_0^{T}\! h(t)h(t)^\mathsf{T} \,\mathrm{d}t\\ \label{eq:gram_bpod}
		&\approx \sum_{i=1}^{N}\delta_i  h_ih_i^\mathsf{H},
	\end{align}
	with $h_i\approx h(t_i),$ $h(t)=e^{At}B$, an end time $T\in \mathbb{R}_{+}$, times $t_1<\dots<t_N\in [0,T]$ and quadrature weights $\delta_i$. The approximation of $h(t_i)$ is done by solving the ODE
	\begin{align}\label{eq:odeh_bpod}
		\frac{\mathrm d}{\mathrm dt}h(t) &= Ah(t), & h(0)=B.
	\end{align}
	From \eqref{eq:gram_bpod} we find that the approximate Cholesky factor is given by 
	\begin{align}
	Z_\text{c}=[h_1, \dots, h_N]\diag(\sqrt{\delta_1}, \dots, \sqrt{\delta_N}).
\end{align}
In \cite[Prop. 2]{Rowley05} it was shown that if approximate Cholesky factors with $\rank(Z_\text{o}^{\mathsf{T}}Z_\text{c})=r$ are used in balanced truncation, then the matrix $V$ from \eqref{eq:WV} contains the first columns of an approximate balancing transformation. It was shown in \cite{Opm12} that for certain quadrature methods for solving \eqref{eq:odeh_bpod} the reduced system obtained by balanced POD matches some moments. We will proceed as in balanced POD to obtain an approximate balancing transformation. To obtain the approximate Cholesky factors of the gramians we solve a system of ODEs which consists of the ODE \eqref{eq:odeh_bpod} for approximating $h(t)$ and a second ODE $\frac{\mathrm d}{\mathrm dt}P(t) = h(t)h(t)^{\mathsf{T}}$ for approximating the time-dependent gramian with Runge-Kutta methods. This allows us to show a connection between the Butcher tableau which characterizes the Runge-Kutta method and the expansion points at which the moments are matched.

\subsection{Rational interpolation}\label{sec:rat_inter}
In rational interpolation \cite{Ant2005,freund00,grimme97} the projection matrices $V$ and $W$ are chosen so that the transfer function $G(s) = C(sI_n-A)^{-1}B$ of the original system \eqref{def:lti} and the transfer function $\hat{G}(s) = \hat C(sI_r-\hat A)^{-1}\hat B$ of the reduced system \eqref{def:redlti} (and some of their derivatives) coincide at certain interpolation points $s\in \mathbb{C}\cup \{\infty\}$.
Rational interpolation is a powerful method: Almost every reduced LTI system \eqref{def:redlti} can be obtained via rational interpolation from \eqref{def:lti}, see \cite{gallivan03}.

A power series expansion around $s_0\in \mathbb{C}\setminus \sigma(A)$ with $\| (s-s_0)(A-s_0I_n)^{-1}\| < 1$ yields
\begin{align}
	G(s) = \sum_{j=0}^{\infty} m_j(s_0)(s-s_0)^j
\end{align}
with the so called moments
\begin{align}
	m_j(s_0) = -C(A-s_0I_n)^{-(j+1)}B = \frac{(-1)^j}{j!}\frac{\mathrm d^{j}}{\mathrm d s^{j}}G(s)\bigg\vert_{s=s_0}.
\end{align}
If either
\begin{align}
	\label{eq:intV}
	\{ (A-s_0I_n)^{-1}B, \dots, (A-s_0I_n)^{-k}B \}&\subseteq \spa V\\
	\label{eq:intW}
	\text{or }\qquad  \{ (A^{\mathsf{T}}-\overline{s_0}I_n)^{-1}C^{\mathsf{T}}, \dots, (A^{\mathsf{T}}-\overline{s_0}I_n)^{-k}C^{\mathsf{T}} \}&\subseteq \spa W
\end{align}
then the first $k$ moments around $s_0$ are matched, i.e. $m_j(s_0)=\hat m_j(s_0)$ for $j=0, \dots, k-1$. If both conditions \eqref{eq:intV} and \eqref{eq:intW} are fulfilled, then even the first $2k$ moments around $s_0$ are matched.

For the expansion point $s_0=\infty$ and $\|s^{-1}A\|<1$ we use the power series expansion
\begin{align}\label{eq:markovexp}
	G(s) = \sum_{j=1}^{\infty}m_j(\infty)s^{-j}
\end{align}
with the Markov parameters $m_j(\infty)=CA^{j-1}B$. If
\begin{align}
	\label{eq:markV}
	\{ B, AB, \dots, A^{k-1}B \}&\subseteq \spa V\\
	\label{eq:markW}
	\text{or }\qquad  \{ C^{\mathsf{T}}, A^{\mathsf{T}}C^{\mathsf{T}}, \dots, (A^{\mathsf{T}})^{k-1}C^{\mathsf{T}} \}&\subseteq \spa W
\end{align}
then the first $k$ Markov parameters are matched, i.e. $m_j(\infty)=\hat m_j(\infty)$ for $j=1, \dots, k$. If both conditions \eqref{eq:markV} and \eqref{eq:markW} are fulfilled, then even the first $2k$ Markov parameters are matched.

The projection matrices can be kept real when the interpolation points occur in conjugated pairs as
	\begin{align}
		&\spa\{(A-s_0I_n)^{-1}v, (A-\overline{s_0}I_n)^{-1}v\} \\
		= &\spa\{\Real((A-s_0I_n)^{-1}v), \Imag((A-s_0I_n)^{-1}v)\}
	\end{align}
holds for real vectors $v$.

Of course combinations of the cases mentioned above and different expansion points are possible. To obtain a well approximating reduced system the choice of the expansion points is essential and many strategies exist to obtain them, see e.g. \cite[Sec. 2.2.2]{BauBF14}.

\subsection{Organization of paper}\label{sec:orga}
In the following we focus on the approximation of the controllability gramian \eqref{def:gramianP} by approximately solving the Lyapunov equation
\begin{align}
	A\mathcal{P} + \mathcal{P}A^{\mathsf{T}} + BB^{\mathsf{T}} = 0.
	\label{eq:lyap}
\end{align}
The observability gramian \eqref{def:gramianQ} satisfies the Lyapunov equation
\begin{align}
	A^{\mathsf{T}}\mathcal{Q} + \mathcal{Q}A + C^{\mathsf{T}}C = 0.
	\label{eq:lyapQ}
\end{align}
It can be treated with the same methods as the controllability gramian by exchanging $A$ and $B$ with $A^{\mathsf{T}}$ and $C^{\mathsf{T}}$, so large parts of our discussion focus on the controllability gramian only.

This paper is organized as follows. In \Cref{sec:gqa} numerical integration with Runge-Kutta methods is introduced and applied to an ODE derived from the time-dependent gramian. It is illustrated how the resulting system is solved efficiently and which space is spanned by the iterates. The numerical solution of the ODE is used for approximate balancing in \Cref{sec:abt}. Using the results from the previous section it is proven that hereby moment matching is performed. In \Cref{sec:connect} we illustrate connections to balanced POD and the ADI iteration. Finally, in \Cref{sec:examples} some examples illustrate our findings.

\section{Gramian quadrature algorithm}\label{sec:gqa}
We now present a quadrature algorithm to obtain approximate Cholesky factors of the gramians. It was first introduced in \cite{bertram19arxiv} and is recapitulated here in concise form. Consider the system of ordinary differential equations
\begin{align}\label{eq:dP}
	\frac{\mathrm d}{\mathrm dt}P(t) &= h(t)h(t)^{\mathsf{T}}, & P(0)=0\in\mathbb{R}^{n\times n},\\
	\frac{\mathrm d}{\mathrm dt}h(t) &= Ah(t), & h(0)=B\in \mathbb{R}^{n\times 1}.\label{eq:dh}
\end{align}
It is solved by the time dependent gramian $P(t)$ given by 
\begin{align}
	P(t) &= \int_0^t e^{A\tau}BB^{\mathsf{T}}e^{A^{\mathsf{T}}\tau}\,\mathrm d\tau = \int_0^t h(\tau)h(\tau)^{\mathsf{T}}\,\mathrm d\tau 
\end{align}
and $h(t) = e^{At}B.$
We intend to solve the above system of ODEs numerically to obtain an approximation to the gramian $\mathcal{P}=\lim_{t\to\infty}P(t)$.

\subsection{Approximating the gramian via Runge-Kutta methods}\label{sec:quad}
There are numerous methods for the numerical solution of ordinary differential equations of the type $\frac{\mathrm d}{\mathrm dt}y(t) = f(t,y(t))$. Single-step methods make use of the fact that
\begin{align}
	y(t_j)&=y(t_{j-1})+\int_{t_{j-1}}^{t_j} f(t,y(t))\,\mathrm dt
\end{align}
holds in order to compute approximate solutions $y_j \approx y(t_j)$ iteratively. Here we consider $s$-stage Runge-Kutta methods (see, e.g., \cite{But16,HaiLW06,HaiNW93,HaiW96}), a particular family of single-step methods. They are defined via
	\begin{align}
		\label{eq:runge-kutta}
		y_j &= y_{j-1} + \omega_j\sum_{i=1}^{s}\beta_ik_i^{(j)}, & j &= 1, \ldots, N,\\
		\label{eq:runge-kutta-k}
		k_i^{(j)} &= f\Big(t_{j-1}+\gamma_i\omega_{j},\, y_{j-1}+\omega_j\sum_{\ell=1}^{s}\lambda_{i\ell}k_\ell^{(j)}\Big),\ & i&=1,\dots,s,
	\end{align}
for certain $\beta_i \in \mathbb{C},$ $\gamma_i  \in \mathbb{R}$, $i = 1, \ldots, s$  and $\lambda_{i\ell} \in \mathbb{C}$, $i,\ell = 1, \ldots, s.$ 
Please note that we allow for complex-valued $\lambda_{ij}$ and $\beta_i$ unlike the usual definition of Runge-Kutta methods.
Moreover,  $\omega_j\coloneqq t_j-t_{j-1} > 0$, $j = 1, \ldots, N,$ denotes the time step size. Often Runge-Kutta methods are given in short hand by the so called Butcher tableau
\begin{align}
	\label{eq:butcher}
\begin{array}{c|c}
\gamma& \Lambda\\
\hline
& \beta^{\mathsf{T}} \\
\end{array}=
\begin{array}{c|cccc}
\gamma_1    & \lambda_{11} & \lambda_{12}& \dots & \lambda_{1s}\\
\gamma_2    & \lambda_{21} & \lambda_{22}& \dots & \lambda_{2s}\\
\vdots & \vdots & \vdots& \ddots& \vdots\\
\gamma_s    & \lambda_{s1} & \lambda_{s2}& \dots & \lambda_{ss} \\
\hline
       & \beta_1    & \beta_2   & \dots & \beta_s\\
\end{array}
\end{align}
with $\Lambda\in\mathbb{C}^{s\times s}$, $\beta \in \mathbb{C}^s$ and $\gamma\in \mathbb{R}^{s}.$

The most involved part in the iteration is the calculation of $k_i^{(j)}$ in \eqref{eq:runge-kutta-k}. If in the Butcher tableau $\Lambda$ is a strict lower triangular matrix, then the $k_i^{(j)}$ can be calculated explicitly one after another and the resulting method is called an explicit Runge-Kutta method.
Otherwise they are only defined implicitly and a system of (in general nonlinear) equations with $sn$ unknowns has to be solved to obtain them. One strategy to simplify the computation is by using lower triangular matrices $\Lambda$, resulting in so called diagonally implicit Runge-Kutta (DIRK) methods. Another kind of methods, derived from DIRK methods, are the Rosenbrock-Wanner methods. There, the nonlinear function $f$ is approximated by a linear function. If the function $f$ to be integrated is linear, then the Rosenbrock-Wanner methods coincide with Runge-Kutta methods.

The ODEs \eqref{eq:dP} and \eqref{eq:dh} are solved with two possibly different $s$-stage Runge-Kutta methods as suggested in \cite[Remark 1]{bertram19arxiv}. The ODE \eqref{eq:dP} is solved with a method based on a Butcher tableau with $\tilde\Lambda\in \mathbb{C}^{s\times s}$ and $\tilde\beta\in \mathbb{R}_{\geq 0}^{s}$. We only allow nonnegative real entries in $\tilde\beta$ to ensure that the approximation to the gramian is positive semidefinite, cf. \eqref{eq:Zj}. The ODE \eqref{eq:dh} is solved using Butcher tableaus with $\Lambda\in \mathbb{C}^{s\times s}$ and $\beta\in \mathbb{C}^{s}$. Herewith we obtain the iteration
\begin{align}\label{eq:RK_mat}
\begin{split}
P_j &= P_{j-1} + \omega_j \sum_{i=1}^s \tilde\beta_i \mathfrak{h}_i^{(j)} (\mathfrak{h}_i^{(j)})^\mathsf{H}, \qquad j = 1, \ldots, N\\
h_j &= h_{j-1} + \omega_j \sum_{i=1}^s \beta_i k_i^{(j)}
\end{split}
\end{align}
with initial values $P_0 = 0 \in \mathbb{R}^{n \times n},$  $h_0 = B \in \mathbb{R}^{n\times 1}$ and
\begin{align}
\mathfrak{h}_i^{(j)} &= h_{j-1} + \omega_j \sum_{\ell=1}^s \lambda_{i\ell} k_\ell^{(j)}=h_{j-1} + \omega_j \sum_{\ell=1}^s \lambda_{i\ell}A\mathfrak{h}_\ell^{(j)},\\
\label{eq:kAh}
k_i^{(j)} &= Ah_{j-1} + \omega_j \sum_{\ell=1}^s \lambda_{i\ell}Ak_\ell^{(j)}
= A\mathfrak{h}_i^{(j)}
\end{align}
for $i =1, \ldots, s.$ Let $\mathcal{H}_{j}=[\mathfrak{h}_1^{(j)},\ldots,\mathfrak{h}_s^{(j)}]\in \mathbb{C}^{n\times s}$, $K_j=[k_1^{(j)}, \dots, k_{s}^{(j)}]\in \mathbb{C}^{n\times s}$. Using this notation and $K_j=A\mathcal{H}_j$ from \eqref{eq:kAh} the above iteration reads
\begin{align}\label{eq:RK_mat_short}
\begin{split}
P_j &= P_{j-1} + \mathcal{H}_j \diag(\omega_j \tilde\beta) \mathcal{H}_j^\mathsf{H},\\
h_j &= h_{j-1} + \omega_j  A\mathcal{H}_j\beta\\
&= h_{j-1} + \omega_j  K_j\beta.
\end{split}
\end{align}
Here the diagonal matrix $\diag(\omega_j \tilde\beta) \in \mathbb{R}_{\geq 0}^{s \times s}$ has the diagonal entries $\omega_j\beta_i$, $i=1, \ldots, s$, and
\begin{align}
	\label{eq:Hj}
	\mathcal{H}_j &= h_{j-1}\otimes \mathds{1}_s^{\mathsf{T}} + \omega_j A\mathcal{H}_j\Lambda^\mathsf{T}
\end{align}
for $j = 1, \ldots, N$ where $\mathds{1}_s=[1, \dots, 1]^{\mathsf{T}}$ is the $s$-dimensional vector containing only ones. For the iteration first $\mathcal{H}_j$ is determined from \eqref{eq:Hj}, then $h_j$ and $P_j$ are computed.

In order to see when $\mathcal{H}_j$ is uniquely determined, \eqref{eq:Hj} is reformulated via vectorization as a linear system of equations with a system matrix of size $ns \times ns$
\begin{align}
	\label{eq:Kvec}
	\left(I_{ns} -\omega_j(\Lambda\otimes A)\right)\vect(\mathcal{H}_{j}) &= {h}_{j-1}\otimes \mathds{1}_s\in\mathbb{C}^{ns\times 1}.
\end{align}
Let $\mu_1, \dots, \mu_s$ and $\lambda_1, \dots, \lambda_n$ be the eigenvalues of $\Lambda$ and $A$ respectively. Then the eigenvalues of $I_{ns} -\omega_j(\Lambda\otimes A)$ are given by $1 - \omega_j \mu_p \lambda_q$, $p = 1, \ldots, s$, $q = 1, \ldots, n$. Thus the solution of \eqref{eq:Kvec} is unique if and only if
\begin{align}\label{eq:eigcond}
	\mu_p\neq \frac{1}{\omega_j\lambda_q}
\end{align}
for all $p = 1, \ldots, s$ and $q = 1, \ldots, n$. 

As $\tilde\beta\in \mathbb{R}_{\geq 0}^{s}$ the approximant $P_j$ is by construction a positive semidefinite matrix and can be expressed as $P_j = Z_jZ_j^{\mathsf{H}}$ for some complex valued matrix $Z_j.$ Thus we have
\begin{align}\label{eq:ZZH}
Z_j Z_j^{\mathsf{H}} &= Z_{j-1}Z_{j-1}^{\mathsf{H}} + \mathcal{H}_j \diag(\omega_j \tilde\beta) \mathcal{H}_j^\mathsf{H}\\
&= \left[ Z_{j-1},  \mathcal{H}_j \diag(\omega_j \tilde\beta)^{\frac{1}{2}} \right]
\left[ Z_{j-1},  \mathcal{H}_j \diag(\omega_j \tilde\beta)^{\frac{1}{2}} \right]^\mathsf{H}. \nonumber
\end{align}
Instead of iterating on $P_j$ as in \eqref{eq:RK_mat_short}, the above observation allows us to iterate on the low rank factor
\begin{align}\label{eq:Zj}
	Z_j &= [Z_{j-1}, \mathcal{H}_j\diag(\omega_j\tilde\beta)^{\frac{1}{2}}] \in \mathbb{C}^{n \times js}
\end{align}
which gains $s$ additional columns in every iteration step.

The procedure to obtain the gramian approximation described in this section is summarized in \Cref{alg:quad}. We require that the eigenvalues of $\Lambda$ satisfy \eqref{eq:eigcond} in order to ensure that all linear system solves have a unique solution and $\tilde\beta\in \mathbb{R}_{\geq 0}^{s}$ to ensure $P_j$ is positive semidefinite.

\begin{algorithm}[ht]
\caption{Approximate Cholesky factor computation via an $s$-stage Runge-Kutta method}
    \label{alg:quad}
    \begin{algorithmic}[1]
	    \Input $A \in \mathbb{R}^{n \times n}$ stable, $B \in \mathbb{R}^{n\times 1}$, positive time step sizes $\left\{ \omega_1,\dots,\omega_{N} \right\}$, Butcher tableau with $\tilde\beta\in \mathbb{R}_{\geq 0}^{s}$ and Butcher tableau with $\Lambda\in \mathbb{C}^{s\times s},\, \beta\in \mathbb{C}^{s}$ which satisfies \eqref{eq:eigcond}
	    \Output  $Z\in \mathbb{C}^{n\times sN}$ with $Z Z^{\mathsf{H}}\approx \mathcal{P}$
    \State initialize $h_0=B$, $Z_0=[\ ]$
    \For {$j=1,\dots,N$}
    \State solve $\mathcal{H}_j=[h_{j-1}, \ldots, h_{j-1}] + \omega_j A\mathcal{H}_j\Lambda^{\mathsf{T}}$ for $\mathcal{H}_j \in \mathbb{C}^{n \times s}$
    \State update 
$Z_j=[Z_{j-1}, \mathcal{H}_j\diag(\omega_j\tilde\beta)^{\frac{1}{2}}]$
\State $h_{j} = h_{j-1} + \omega_jA\mathcal{H}_j\beta$
    \EndFor
    \State $Z=Z_N$
    \end{algorithmic}
\end{algorithm}

\subsection{Computation of $\mathcal{H}_j$ in \Cref{alg:quad}}\label{sec:compHj}
The main part of \Cref{alg:quad} is solving for $\mathcal{H}_j$ in step 3. Of course \eqref{eq:Kvec} can be used to determine $\mathcal{H}_j$. However, this means the solution of the $ns$-dimensional system \eqref{eq:Kvec}. Here we present a novel, more efficient way to obtain $\mathcal{H}_j$ with the solution of $s$ only $n$-dimensional linear systems.

Let $(\Lambda')^{\mathsf{T}}=S\Lambda^{\mathsf{T}} S^{-1}\in \mathbb{C}^{s\times s}$ be a Schur decomposition of $\Lambda^{\mathsf{T}}$, so the diagonal entries of the upper triangular matrix $(\Lambda')^{\mathsf{T}}$ are the eigenvalues $\mu_1, \dots, \mu_s$ of $\Lambda$. Consider \eqref{eq:Hj} and define $\mathcal{H}_j'=[{\mathfrak{h}_1'}^{(j)}, \dots, {\mathfrak{h}_s'}^{(j)}]$ via $\mathcal{H}_j=\mathcal{H}_j'S$. Then \eqref{eq:Hj} can be reformulated as 
\begin{align}\label{eq:Hdash}
	\mathcal{H}_j' &= (h_{j-1}\otimes \mathds{1}_s^{\mathsf{T}})S^{-1} + \omega_jA\mathcal{H}_j'(\Lambda')^{\mathsf{T}}.
\end{align}
Let $[\alpha_1, \dots, \alpha_s]=\mathds{1}_s^{\mathsf{T}} S^{-1}$ be the row vector containing the column sums of $S^{-1}$. Then we can rewrite \eqref{eq:Hdash} as
\begin{align}\label{eq:Hdashalpha}
	\mathcal{H}_j' &= [\alpha_1 h_{j-1}, \dots, \alpha_s h_{j-1}] + \omega_jA\mathcal{H}_j'(\Lambda')^{\mathsf{T}}.
\end{align}
To obtain $\mathcal{H}_j'$, the following systems of linear equations have to be solved 
\begin{align}
	(I_n-\omega_j\mu_iA)\mathfrak{h}_i'^{(j)} = \alpha_ih_{j-1} + \omega_j\sum_{l=1}^{i-1}\lambda_{il}'A{\mathfrak{h}_l'}^{(j)}
	\label{eq:Hj_schur}
\end{align}
for $i=1, \dots, s$. Finally, $\mathcal{H}_j$ is assembled via $\mathcal{H}_j=\mathcal{H}_j'S$. This procedure with the Schur decomposition reduces the effort from solving one $ns$-dimensional system \eqref{eq:Kvec} to the solution of $s$ systems of dimension $n$ in \eqref{eq:Hj_schur} and one Schur decomposition of size $s$.

\subsection{The space spanned by the approximate Cholesky factor $Z$}
The main result of this section is that the columns of the approximate Cholesky factor $Z=Z_N$ obtained from \Cref{alg:quad} span a (rational) Krylov subspace which is essentially determined by the eigenvalues of $\omega_i\Lambda$. To show this we first reveal how the iterate $Z$ can be obtained in only one step of \Cref{alg:quad} with certain Butcher tableaus assembled from $\Lambda,\, \beta,\, \tilde\beta$ and the time step sizes $\omega_j$.

After $N$ steps of \Cref{alg:quad} we find the approximate Cholesky factor $Z$ which is recursively defined via step 4. Expanding the for loop
\begin{align}
	\label{eq:Zexpanded}
	Z = [\mathcal{H}_1, \dots, \mathcal{H}_N] \diag(\vect(\omega_1\tilde\beta, \dots, \omega_N\tilde\beta))^{\frac{1}{2}}
\end{align}
is obtained. For $\mathcal{H}_1$ we have from step 3 of \Cref{alg:quad}
\begin{align}
	\label{eq:H1}
	\mathcal{H}_1 &= \mathds{1}_s^{\mathsf{T}} \otimes h_0 + \omega_1A\mathcal{H}_1\Lambda^{\mathsf{T}}\\
	&= \mathds{1}_s^{\mathsf{T}}\otimes h_0 + A\mathcal{H}_1(\omega_1\Lambda^{\mathsf{T}}).
\end{align}
For $\mathcal{H}_2$ we find with step 3 and step 5 of \Cref{alg:quad}
\begin{align}
	\label{eq:rewriteH2}
	\mathcal{H}_2 &= \mathds{1}_s^{\mathsf{T}}\otimes h_1 + \omega_2A\mathcal{H}_2\Lambda^{\mathsf{T}}\\
	&= \mathds{1}_s^{\mathsf{T}}\otimes( h_0 +  \omega_1A\mathcal{H}_1\beta) + \omega_2A\mathcal{H}_2\Lambda^{\mathsf{T}}\\
	&= \mathds{1}_s^{\mathsf{T}}\otimes h_0 + A\mathcal{H}_1(\omega_1[\beta,\dots,\beta]) + A\mathcal{H}_2(\omega_2\Lambda^{\mathsf{T}})\\
	&= \mathds{1}_s^{\mathsf{T}}\otimes h_0 + A[\mathcal{H}_1, \mathcal{H}_2]\begin{bmatrix}\omega_1[\beta,\dots,\beta]\\ \omega_2\Lambda^{\mathsf{T}}\end{bmatrix}.
\end{align}
Putting $\mathcal{H}_1$ from \eqref{eq:H1} and $\mathcal{H}_2$ from \eqref{eq:rewriteH2} together, one yields
\begin{align}
	\label{eq:rewriteH1H2}
	[\mathcal{H}_1, \mathcal{H}_2] &= \mathds{1}_{2s}^{\mathsf{T}}\otimes h_0 + A[\mathcal{H}_1, \mathcal{H}_2]\begin{bmatrix}\omega_1\Lambda^{\mathsf{T}} & \omega_1[\beta,\dots,\beta]\\ 0 & \omega_2\Lambda^{\mathsf{T}}\end{bmatrix}.
\end{align}
Proceeding in this way up to iteration step $N$ and setting $\mathcal{\hat H}=[\mathcal{H}_1,\dots, \mathcal{H}_N]$ this leads to the equation
\begin{align}
	\label{eq:rewriteH1HN}
	\mathcal{\hat H} = \mathds{1}_{Ns}^{\mathsf{T}}\otimes h_0 + A\mathcal{\hat H}\hat\Lambda^{\mathsf{T}}
\end{align}
with 
\begin{align}\label{eq:Lhat}
	\hat\Lambda^{\mathsf{T}} \coloneqq 
\begin{bmatrix}
	\omega_1\Lambda^{\mathsf{T}} & \omega_1[\beta,\dots,\beta] & \cdots & \omega_1[\beta,\dots,\beta]\\
		0 & \omega_2\Lambda^{\mathsf{T}} & \omega_2[\beta,\dots,\beta] & \omega_2[\beta,\dots,\beta]\\
		\vdots & 0 & \ddots &\vdots\\
		0& \cdots & 0  &\omega_N\Lambda^{\mathsf{T}}
	\end{bmatrix}\in \mathbb{C}^{Ns\times Ns}.
\end{align}
Thus, the result $Z$ from \eqref{eq:Zexpanded} can also be interpreted as one step of \Cref{alg:quad} with time step size $\hat\omega_1=1$, $\tilde{\hat\beta}=\vect([\omega_1\tilde\beta, \dots, \omega_N\tilde\beta])$ and $\hat\Lambda$ from \eqref{eq:rewriteH1HN}. It is therefore sufficient to analyze one step of \Cref{alg:quad}. The situation with more than one step is contained as a special case as described above.

Let all entries of $\tilde\beta$ be positive, i.e. $\tilde\beta\in \mathbb{R}_{+}^{s}$, then the diagonal matrix in \eqref{eq:Zexpanded} is regular and so the space spanned by the columns of $Z$ equals the one spanned by the columns of $\mathcal{\hat H}$. We proceed with similarity transformations of $\hat\Lambda^{\mathsf{T}}$ as in \Cref{sec:compHj} to uncouple the columns of $\mathcal{\hat H}$. Define $\mathcal{\hat H}=\mathcal{\hat H}'S$ with a similarity transformation $S\in \mathbb{C}^{Ns\times Ns}$ which transforms $\hat\Lambda^{\mathsf{T}}$ to its Jordan canonical form 
\begin{align}\label{eq:Lhatjcf}
(\hat\Lambda')^{\mathsf{T}}=S\hat\Lambda^{\mathsf{T}}S^{-1}
	=\begin{bmatrix}
		J_1 &  & \\
		& \ddots & \\
		&  & J_q
	\end{bmatrix}
\end{align}
with $q$ Jordan blocks $J_l\in \mathbb{C}^{s_l\times s_l}$ of dimension $s_l$ for $l=1, \dots, q$. We further partition $\mathcal{\hat H}'=[\mathcal{\hat H}_{1}', \dots, \mathcal{\hat H}_{q}']$ and 
\begin{align}\label{eq:alphal}
	\mathds{1}^{\mathsf{T}}_{Ns}S^{-1}=[\alpha^{(1)}, \dots, \alpha^{(q)}]
\end{align}
according to the sizes of the Jordan blocks, i.e. $\mathcal{\hat H}_{l}'\in \mathbb{C}^{n\times s_l}$ and $(\alpha^{(l)})^{\mathsf{T}}\in \mathbb{C}^{s_l}$. Multiplication of \eqref{eq:rewriteH1HN} with $S^{-1}$ from the right yields
\begin{align}\label{eq:hatHdash}
	[\mathcal{\hat H}_{1}', \dots, \mathcal{\hat H}_{q}'] &= [\alpha^{(1)}, \dots, \alpha^{(q)} ]\otimes h_0 + A[\mathcal{\hat H}_{1}', \dots, \mathcal{\hat H}_{q}']
	\begin{bmatrix}
		J_1 &  & \\
		& \ddots & \\
		&  & J_q
	\end{bmatrix}.
\end{align}
Due to the partitioning this equation is equivalent to
\begin{align}\label{eq:hatHdashl}
	\mathcal{\hat H}_{l}' &= \alpha^{(l)}\otimes h_0 + A\mathcal{\hat H}_{l}'J_l \quad \text{ for } l=1, \dots, q.
\end{align}
The matrices $\mathcal{\hat H}_{l}'=[\mathfrak{\hat h}_1'^{(l)}, \dots, \mathfrak{\hat h}_{s_l}'^{(l)}]$ are determined by
\begin{align}\label{eq:jordanstep}
	\begin{split}
	(I_n-\hat\mu_l A)\mathfrak{\hat h}'^{(l)}_{1} &= \alpha_{1}^{(l)} h_{0},\\
	(I_n-\hat\mu_l A)\mathfrak{\hat h}'^{(l)}_{i} &= \alpha_{i}^{(l)} h_{0} + A \mathfrak{\hat h}'^{(l)}_{i-1} \quad \text{ for } i=2, \dots, s_l
	\end{split}
\end{align}
with the eigenvalue $\hat\mu_l$ of $\hat\Lambda$ as the diagonal element of the Jordan block $J_l$.

Before we proceed with the main result of this section we state a technical lemma.
\begin{lemma}\label{lem:alphaone}
	Let $(\hat\Lambda^{\mathsf{T}},\mathds{1}_{Ns}^{\mathsf{T}})$ be observable. Then the transformation matrix $S$ to Jordan canonical form in \eqref{eq:Lhatjcf} can be chosen such that $\alpha^{(l)}=[1,0,\cdots,0]$ holds for $l=1, \cdots, q$ in \eqref{eq:alphal}.
\end{lemma}
\begin{proof}
	For $l=1, \cdots, q$ define $e_l=[1, 0 \cdots, 0]\in \mathbb{R}^{1\times s_l}$. Assume that there exist polynomials $p_l$ with
\begin{align}\label{eq:polycond}
	\alpha^{(l)}=e_lp_l(J_l).
\end{align}
Now replace the matrix $S$ in \eqref{eq:Lhatjcf} and \eqref{eq:alphal} with $\tilde S=\diag(p_1(J_1), \cdots, p_{q}(J_q))S$. As $J_l$ commutes with rational functions in $J_l$ the matrix $\tilde S$ is a similarity transformation to Jordan canonical form, too, and it holds
	\begin{align}
		\mathds{1}^{\mathsf{T}}_{Ns}\tilde S^{-1}&=\mathds{1}^{\mathsf{T}}_{Ns}S^{-1}\diag(p_1(J_1), \cdots, p_{q}(J_q))^{-1}\\
		&=[\alpha^{(1)}, \dots, \alpha^{(q)} ]\diag(p_1(J_1)^{-1}, \cdots, p_{q}(J_q)^{-1})\\
		&=[e_1, \cdots, e_q].
	\end{align}

	It remains to show that a polynomial $p_l$ fulfilling \eqref{eq:polycond} exists and $p_l(J_l)$ is invertible for $l=1, \cdots, q$. Define the upper shift matrix $r_l(J_l)= -\hat\mu_lI + J_l$ with ones above the diagonal and zeros everywhere else. It holds $e_lr_l(J_l)^{i-1}=[0, \cdots, 0, 1, 0, \cdots 0]$, a vector with a one at position $i$ for $i=1, \cdots, s_l$. For the $i$th row of $p_l(J_l)$ we find with \eqref{eq:polycond}
	\begin{align}
		[0, \cdots, 0, 1, 0, \cdots 0]p_l(J_l) &= e_lr_l(J_l)^{i-1}p_l(J_l)\\
		&= e_lp_l(J_l)r_l(J_l)^{i-1}\\
		&= \alpha^{(l)}r_l(J_l)^{i-1}\\
		&= [0, \cdots, 0, \alpha^{(l)}_1, \cdots, \alpha^{(l)}_{s_l-(i-1)}].
	\end{align}
	This implies that $p_l(J_l)$ is an upper triangular matrix with entries $\alpha^{(l)}_1$ on the diagonal. As $(\hat\Lambda^{\mathsf{T}},\mathds{1}_{Ns}^{\mathsf{T}})$ is observable, so is $(J_l,\alpha^{(l)})$ and thus $\alpha^{(l)}_1\neq 0$. So $p_l(J_l)$ is invertible, which concludes the proof.
\end{proof}

These preparations allow us to state the following lemma.
\begin{lemma}\label{lem:span}
	Let $Ns<n$ and $(\hat\Lambda^{\mathsf{T}},\mathds{1}_{Ns}^{\mathsf{T}})$ be observable. If $\hat\mu_l\neq 0$ then
	\begin{align}\label{eq:spanH}
		\spa\mathcal{\hat H}'_l &= \spa\!\left\{ (I_n-\hat\mu_l A)^{-i}h_0\mid i=1, \dots, s_l\right\}.
	\end{align}
	If $\hat\mu_l = 0$ then
	\begin{align}\label{eq:spanH0}
		\spa\mathcal{\hat H}'_l &= \spa\!\left\{ A^{i}h_0\mid i=0, \dots, s_l-1\right\}.
	\end{align}
\end{lemma}
\begin{proof}
	In this proof set $\mathfrak{\hat h}_i' \coloneqq \mathfrak{\hat h}_i'^{(l)}$ for better readability. Due to the observability of $(\hat\Lambda^{\mathsf{T}},\mathds{1}_{Ns}^{\mathsf{T}})$ we find from \eqref{eq:Lhatjcf} and \eqref{eq:alphal} that $(J_l,\alpha^{(l)})$ is observable. Due to \Cref{lem:alphaone} we can assume $\alpha^{(l)}=[1,0,\dots,0]$.

	Let $\hat\mu_l\neq 0$. Because of \eqref{eq:jordanstep}
	\begin{align}
		\spa \mathfrak{\hat h}'_{1} &= \spa\{(I_n-\hat\mu_l A)^{-1} h_{0}\}
	\end{align}
	holds. From \eqref{eq:jordanstep} we find for $1 < i\leq s_l$ as $\alpha^{(l)}_i=0$ 
\begin{align}
	\mathfrak{\hat h}_{i}' &= (I_n-\hat\mu_lA)^{-1}A\mathfrak{\hat h}_{i-1}'\\
	&= (I_n-\hat\mu_lA)^{-1}(-\hat\mu_l^{-1}(I_n -\hat\mu_l A) +\hat\mu_l^{-1}I_n)\mathfrak{\hat h}_{i-1}'\\
	&= - \hat\mu_l^{-1}\mathfrak{\hat h}_{i-1}' + \hat\mu_l^{-1}(I_n-\hat\mu_lA)^{-1}\mathfrak{\hat h}_{i-1}'.
\end{align}
Via induction this concludes the first part of the proof.

Now let $\hat\mu_l=0$. From \eqref{eq:jordanstep}
\begin{align}
	\spa \mathfrak{\hat h}'_1 = \spa h_0
\end{align}
is immediate. For $1 < i\leq s_l$ we have
\begin{align}
	\mathfrak{\hat h}'_{i} &= A \mathfrak{\hat h}'_{i-1},
\end{align}
and the claim again results from induction.
\end{proof}
We conclude that the space spanned by $\mathcal{\hat H}'$ (and thus also by $\mathcal{\hat H}$) mainly depends on the eigenvalues $\hat\mu_l$ of $\hat\Lambda$ and the dimensions $s_l$ of their eigenspaces.

\section{Approximate balancing transformation}\label{sec:abt}
We now present an algorithm which generates an approximate balancing transformation. The reduced system is obtained via projection using approximated gramians. It can be seen as a variant of balanced POD where the Cholesky factors of the gramians are approximated using the quadrature described in \Cref{sec:quad}. This procedure is summarized in \Cref{alg:abt}.

Note that due to the use of Butcher tableaus with complex entries in general complex reduced system matrices are obtained. This is the reason for using conjugate transposition $\mathsf{H}$ instead transposition $\mathsf{T}$.

\begin{algorithm}[ht]
\caption{Approximate balancing transformation}
    \label{alg:abt}
    \begin{algorithmic}[1]
	    \Input system matrices $A \in \mathbb{R}^{n \times n}$ stable, $B \in \mathbb{R}^{n\times 1}$, $C\in \mathbb{R}^{1\times n}$, positive time step sizes $\left\{ \omega_1,\dots,\omega_{N} \right\}$ and $\left\{ \tau_1,\dots,\tau_{N} \right\}$, Butcher tableaus with $\tilde\beta_{\text{c}},\tilde\beta_{\text{o}}\in \mathbb{R}_{\geq 0}^{s}$ and Butcher tableaus with $\Lambda_{\text{c}},\Lambda_{\text{o}}\in \mathbb{C}^{s\times s},\, \beta_\text{c},\beta_\text{o}\in \mathbb{C}^{s}$ which satisfy \eqref{eq:eigcond}
	    \Output reduced system matrices $\hat A\in \mathbb{C}^{r\times r}$, $\hat B\in \mathbb{C}^{r\times 1}$, $\hat C\in \mathbb{C}^{1\times r}$ with $r=\rank(Z_\text{o}^{\mathsf{H}}Z_\text{c})$
	    \State obtain $Z_\text{c}$ with $Z_\text{c}Z_\text{c}^{\mathsf{H}}\approx \mathcal{P}$ from \Cref{alg:quad} with $A$, $B$, $\Lambda_\text{c}$, $\beta_\text{c}$, $\tilde \beta_\text{c}$ and $\{\omega_1, \dots, \omega_{N}\}$
	    \State obtain $Z_\text{o}$ with $Z_\text{o}Z_\text{o}^{\mathsf{H}}\approx \mathcal{Q}$ from \Cref{alg:quad} with $A^{\mathsf{T}}$, $C^{\mathsf{T}}$, $\Lambda_\text{o}$, $\beta_\text{o}$, $\tilde \beta_\text{o}$ and $\{\tau_1, \dots, \tau_{N}\}$
	    \State calculate compact SVD $Z_\text{o}^{\mathsf{H}}Z_\text{c} =U\Sigma T^\mathsf{H}$
	\State assemble projection matrices $V=Z_{\text{c}}T\Sigma^{-\frac{1}{2}}$, $W=Z_{\text{o}}U\Sigma^{-\frac{1}{2}}$
	\State return $\hat A=W^{\mathsf{H}}AV$, $\hat B=W^{\mathsf{H}}B$, $\hat C=CV$
    \end{algorithmic}
\end{algorithm}

As will be shown next, the transfer function of the reduced system generated by \Cref{alg:abt} interpolates the transfer function of the original system at expansion points which depend on the eigenvalues of the Butcher tableaus and the time step sizes. In particular the expansion points are the inverse eigenvalues of $\omega_i\Lambda_{\text{c}}$ for $i=1, \dots, N_\text{c}$ and the conjugated inverse eigenvalues of $\tau_i\Lambda_\text{o}$ for $i=1, \dots, N_\text{o}$.
\begin{theorem}\label{thm:moma}
	Let the inputs of \Cref{alg:abt} with $\tilde\beta_\text{c}, \tilde\beta_\text{o}\in \mathbb{R}^{s}_+$ be given. Define $\hat\Lambda_{\text{c}}^{\mathsf{T}}$ as in \eqref{eq:Lhat} with $\Lambda_{\text{c}}$, $\beta_{\text{c}}$ and $\{\omega_1, \dots, \omega_{N}\}$. Define $\hat\Lambda_{\text{o}}^{\mathsf{T}}$ as in \eqref{eq:Lhat} with $\Lambda_{\text{o}}$, $\beta_{\text{o}}$ and $\{\tau_1, \dots, \tau_{N}\}$.
	Let $\left\{ \hat\mu_1, \dots, \hat\mu_{q_\text{c}} \right\}= \cup_{i=1}^{N}\sigma(\omega_i\Lambda_{\text{c}})$ and $\left\{\hat\nu_1, \dots, \hat\nu_{q_\text{o}}\right\} = \cup_{i=1}^{N}\sigma(\tau_i\Lambda_{\text{o}})$ be the eigenvalues of $\hat\Lambda_\text{c}$ and $\hat\Lambda_\text{o}$ with multiplicities $s_1, \dots, s_{q_{\text{c}}}$ and $t_1, \dots, t_{q_{\text{o}}}$.

	If $(\hat\Lambda_{\text{c}}^{\mathsf{T}},\mathds{1}^{\mathsf{T}}_{Ns})$ and $(\hat\Lambda_{\text{o}}^{\mathsf{T}},\mathds{1}^{\mathsf{T}}_{Ns})$ are observable and $\rank{Z_\text{o}^{\mathsf{H}}Z_{\text{c}}}=Ns$ holds, then the transfer function of the reduced system with system matrices $\hat A, \hat B, \hat C$ produced by \Cref{alg:abt} satisfies
	\begin{align}\label{eq:interp_mu}
	\begin{split}
		\hat G^{(i)}(\hat\mu_{l_\text{c}}^{-1}) &= G^{(i)}(\hat\mu_{l_\text{c}}^{-1})\quad \text{for } i=0, \dots, s_{l_{\text{c}}}-1,\\
		\hat G^{(i)}(\overline{\hat\nu}_{l_\text{o}}^{-1}) &= G^{(i)}(\overline{\hat\nu}_{l_\text{o}}^{-1})\quad \text{for } i=0, \dots, t_{l_\text{o}}-1
	\end{split}
	\end{align}
	for $l_\text{c}=1, \dots, q_{\text{c}}$ and $l_\text{o}=1, \dots, q_{\text{o}}$. For any zero eigenvalues the corresponding interpolation in \eqref{eq:interp_mu} has to be read as interpolation at $\infty$. If some of the values $\hat\mu_i$ and $\overline{\hat\nu_j}$ coincide, even higher derivatives are interpolated.
\end{theorem}
\begin{proof}
	The reduced system is generated via projection with the matrices $V$ and $W$. Due to step 3 and 4 of \Cref{alg:abt} and as ${Z_\text{o}^{\mathsf{H}}Z_{\text{c}}}$ is regular $\spa(V)=\spa(Z_\text{c})$ and $\spa(W)=\spa(Z_\text{o})$ hold. With \Cref{lem:span} we find for $\hat\mu_{l_\text{c}},\hat\nu_{l_\text{o}}\neq 0$
	\begin{align}
	\begin{split}
		\spa\!\left\{ (I_n-\hat\mu_{l_\text{c}} A)^{-i}B\mid i=1, \dots, s_{l_\text{c}}\right\}\subseteq \spa(V),\\
		\spa\!\left\{ (I_n-\hat\nu_{l_\text{o}} A^{\mathsf{T}})^{-i}C^{\mathsf{T}}\mid i=1, \dots, t_{l_\text{o}}\right\}\subseteq \spa(W).
	\end{split}
	\end{align}
	Due to $(I_n-\hat\mu_{l_\text{c}} A)^{-1}=-\hat\mu_{l_\text{c}}^{-1}(A-\hat\mu_{l_\text{c}}^{-1}I_n)^{-1}$ and $(I_n-\hat\nu_{l_\text{o}} A^{\mathsf{T}})^{-1}=-\hat\nu_{l_\text{o}}^{-1}(A^{\mathsf{T}}-\hat\nu_{l_\text{o}}^{-1}I_n)^{-1}$ this means
	\begin{align}
	\begin{split}
		\spa\!\left\{ (A-\hat\mu_{l_\text{c}}^{-1}I_n)^{-i}B\mid i=1, \dots, s_{l_\text{c}}\right\}\subseteq \spa(V),\\
		\spa\!\left\{ (A^{\mathsf{T}}-\hat\nu_{l_\text{o}}^{-1}I_n)^{-i}C^{\mathsf{T}}\mid i=1, \dots, t_{l_\text{o}}\right\}\subseteq \spa(W).
	\end{split}
	\end{align}
	Further, if $\hat\mu_{l_\text{c}},\hat\nu_{l_\text{o}}=0$, then
	\begin{align}
		\begin{split}
			\spa\!\left\{ A^{i}B\mid i=0, \dots, s_{l_\text{c}}-1\right\}\subseteq \spa(V),\\
			\spa\!\left\{ (A^{\mathsf{T}})^{i}C^{\mathsf{T}}\mid i=0, \dots, t_{l_\text{o}}-1\right\}\subseteq \spa(W).
		\end{split}
	\end{align}
	Due to \Cref{sec:rat_inter} this concludes the proof.
\end{proof}
It is interesting to see that using a Runge-Kutta method it is not possible to match moments around the expansion point zero, as this would require an infinite eigenvalue of $\Lambda$ from the Butcher tableau or an infinite time step size, which is impossible.

In \cite{Opm12} complex time step sizes $\omega_j$ ($\tau_j$ respectively) are used in Runge-Kutta methods to achieve moment matching around complex expansion points. This is unfeasible in the method presented here as then the iterates $P_j$ are in general not positive semidefinite and the approximate Cholesky factors $Z_j$ would not exist. Instead, in the framework presented here, complex tableaus may be used.

\section{Connection to other methods}\label{sec:connect}
We now show the connection of the method presented here to other methods involving gramian approximations with low-rank Cholesky factors. We only consider the controllability gramian $\mathcal{P}$. The approximation of the observability gramian $\mathcal{Q}$ is done analogously, cf. \Cref{sec:orga}. All methods have in common that the approximate Cholesky factors are computed directly, that is, no Cholesky decomposition of a large $n\times n$ matrix is necessary.

\subsection{Balanced POD}
\label{sec:bpod}
We first consider balanced POD as introduced in \cite{Rowley05} and summarized at the end of \Cref{sec:bt}. A central task in BPOD is the numerical solution of the ODE \eqref{eq:odeh_bpod}. Unfortunately in \cite{Rowley05} it is not stated which numerical method should be used for solving the ODE. In the following we assume a Runge-Kutta method with $\Lambda_h$ and $\beta_h$ is used to solve the ODE in the same way as $\eqref{eq:dh}$ was solved in \Cref{sec:quad}. In particular, for $h_0=B$ and time step sizes $\omega_j=t_{j}-t_{j-1}$ this means
	\begin{align}\label{eq:bpod_iteration}
	\mathcal{H}_j &= [h_{j-1}, \dots, h_{j-1}] + \omega_jA\mathcal{H}_{j}\Lambda_{h}^{\mathsf{T}}\\
	h_j &= h_{j-1} + \omega_{j}A\mathcal{H}_j\beta_{h}^{\mathsf{T}}
\end{align}
just as in \Cref{alg:quad}, but in the BPOD method the approximate Cholesky factor is updated via
\begin{align}
	Z_{j} = [Z_{j-1}, h_j\delta_j^{\frac{1}{2}}]
\end{align}
instead of $Z_{j} = [Z_{j-1}, \mathcal{H}_j\diag(\omega_j\tilde\beta)^{\frac{1}{2}}]$ as in \Cref{alg:quad}. We illustrate how the balanced POD iterates can be obtained using \Cref{alg:quad} in case $h_j\delta_j h_j^{\mathsf{H}}$ and $\mathcal{H}_j\diag(\omega_{j}\tilde\beta)\mathcal{H}_{j}^{\mathsf{H}}$ coincide. Due to the dimension of $h_j$ and $\mathcal{H}_j$ this is only possible for Butcher tableaus of size $s=1$ or for $\tilde\beta$ having only one nonzero entry.

We first consider the case $s=1$ and thus have $\mathcal{H}_j\in \mathbb{C}^{n\times 1}$. So \eqref{eq:bpod_iteration} becomes
\begin{align}
	\mathcal{H}_j &= h_{j-1} + \omega_jA\mathcal{H}_{j}\Lambda_{h}^{\mathsf{T}}\\
	h_j &= h_{j-1} + \omega_{j}A\mathcal{H}_j\beta_{h}^{\mathsf{T}},
\end{align}
i.e. $\mathcal{H}_j=h_j$ if $\Lambda_{h}=\beta_h$. This is e.g. fulfilled in the backward Euler method with $\Lambda_h=\beta_h=1$. If additionally $\tilde\beta=\nicefrac{\delta_j}{\omega_j}$, balanced POD and \Cref{alg:quad} produce the same iterates.

In case of arbitrary Butcher tableaus with $s$-dimensional $\Lambda_h$ and $\beta_h$ the way BPOD fits into the framework presented here is rather crude. Consider a Butcher tableau with the $s+1$-dimensional matrices
\begin{align}
	\Lambda = \begin{bmatrix}\Lambda_h& 0\\ \beta_h^{\mathsf{T}}& 0\end{bmatrix},\ \beta=\begin{bmatrix}\beta_h\\ 0\end{bmatrix},\ \tilde\beta = \begin{bmatrix} 0\\ \nicefrac{\delta_j}{\omega_j}\end{bmatrix}.
\end{align}
\Cref{alg:quad} generates the iterate
\begin{align}
	\underbrace{[\mathfrak{h}_1^{(j)}, \dots, \mathfrak{h}_s^{(j)}, \mathfrak{h}_{s+1}^{(j)}]}_{=\mathcal{H}_j} &= [h_{j-1}, \dots, h_{j-1}] + \omega_jA[\mathfrak{h}_1^{(j)}, \dots, \mathfrak{h}_s^{(j)}, \mathfrak{h}_{s+1}^{(j)}]\begin{bmatrix}\Lambda_h^{\mathsf{T}}& \beta_h\\ 0& 0\end{bmatrix}.
\end{align}
Separating the first $s$ columns from the last one yields
\begin{align}
	[\mathfrak{h}_1^{(j)}, \dots, \mathfrak{h}_s^{(j)}] &= [h_{j-1}, \dots, h_{j-1}] + \omega_jA[\mathfrak{h}_1^{(j)}, \dots, \mathfrak{h}_s^{(j)}]\Lambda_h^{\mathsf{T}}\\
	\mathfrak{h}_{s+1}^{(j)}&= h_{j-1} + \omega_jA[\mathfrak{h}_1^{(j)}, \dots, \mathfrak{h}_s^{(j)}]\beta_h
\end{align}
and so $h_j=\mathfrak{h}_{s+1}$. Due to the zero entries in $\tilde\beta$ we further find
\begin{align}
	\mathcal{H}_j\diag(\omega_j\tilde\beta)\mathcal{H}_j^{\mathsf{H}} &= \mathfrak{h}_j\omega_j\frac{\delta_j}{\omega_j}\mathfrak{h}_j^{\mathsf{H}}\\
	&= h_j\delta_j h_j^{\mathsf{H}}
\end{align}
i.e. \Cref{alg:quad} and BPOD produce the same iterates for this special choice of tableaus.

\subsection{The ADI iteration}\label{sec:adi}
It was shown in \cite{bertram19arxiv} that for certain Butcher tableaus \Cref{alg:quad} is equivalent to the ADI iteration \cite{peaceman55,lu91,li2002,Kue16,wolf13}. In particular, the gramian approximation produced by \Cref{alg:quad} for Butcher tableaus with $\beta=\tilde\beta$ and $\Lambda$ satisfying
\begin{align}\label{eq:MNull}
	\diag(\beta)\overline\Lambda + \Lambda^{\mathsf{T}}\diag(\beta) -\beta\beta^{\mathsf{T}} = 0
\end{align}
is equivalent to ADI approximants with parameters which are the negative inverses of the eigenvalues of $\omega_i\Lambda$. Runge-Kutta methods which fulfill \eqref{eq:MNull} are given by the family of Gauß-Legendre methods (see \cite{bertram19arxiv}, \cite[Lem. 5.3]{iserles_2008}), i.e. the implicit midpoint rule with
\begin{align}\label{eq:gaussleg1}
	\Lambda = \frac{1}{2},\ \beta=1
\end{align}
or the Gauß-Legendre method with $s=2$ as in \eqref{eq:gaussleg2}.
A more generic way to construct Butcher tableaus which satisfy \eqref{eq:MNull} is given by the lower triangular matrices
\begin{align}
	\label{eq:DIRKtableaus}
	\Lambda=\begin{bmatrix}\mu_1 &0&\cdots&0\\ 2\Real(\mu_1) & \mu_2 & \ddots & 0\\ \vdots& \vdots& \ddots&\vdots\\2\Real(\mu_1)& 2\Real(\mu_2)& \cdots & \mu_s\end{bmatrix},\ \beta=\begin{bmatrix} 2\Real(\mu_1)\\ 2\Real(\mu_2)\\ \vdots \\2\Real(\mu_s)\end{bmatrix}
\end{align}
with parameters $\mu_1, \dots, \mu_s\in \mathbb{C}_{+}$. With this tableau the connection to the ADI parameters is immediate as the eigenvalues can be read off the diagonal. An ADI iteration with parameters $\alpha_i\in \mathbb{C}_{-}$ is thus equivalent to one step of \Cref{alg:quad} with step size $\omega_1=1$ using a Butcher tableau given by \eqref{eq:DIRKtableaus} with $\mu_i=-\alpha_i^{-1}$, see \cite[Thm.~4]{bertram19arxiv}.
From \Cref{lem:span} and \Cref{thm:moma} it follows that the ADI iterates span a rational Krylov space and, if used in \Cref{alg:abt}, the moments at $-\alpha_i=\mu_i^{-1}$ are matched. See also \cite[Section 2.4]{BauBF14} for a different proof.

\section{Examples}\label{sec:examples}
In this section we illustrate the findings from \Cref{thm:moma}. We state the expansion points at which moments are matched for certain Runge-Kutta methods and visualize them in the complex plane.

Explicit Runge-Kutta methods are parameterized by Butcher tableaus with strictly lower triangular $\Lambda$. As such matrices have just zero eigenvalues only moments around $\infty$ are matched for explicit methods. An example is Euler's method given by the Butcher tableau with $\Lambda=0,$ $\beta=1$.

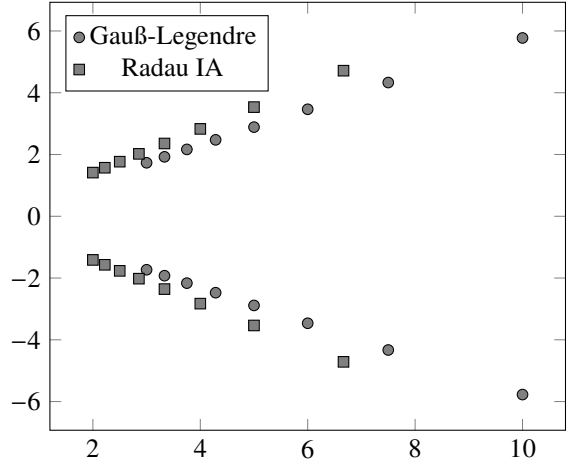
\begin{figure}
	\begin{center}
\begin{tikzpicture}
	\begin{axis}[cycle list name=black white,
		legend pos=north west]
		\addplot+[only marks] table[x=gaussr,y=gaussi] {expansion_points.dat};
		\addplot+[only marks] table[x=radaur,y=radaui] {expansion_points.dat};
		\legend{Gauß-Legendre, Radau IA}
	\end{axis}
\end{tikzpicture}
\caption{Expansion points in the complex plane for Gauß-Legendre and Radau IA method.}\label{fig:exp_points}
\end{center}
\end{figure}

For the backward Euler method we have $\Lambda=1,$ $\beta=1$, so the moments are matched around the inverse time step sizes $\omega_j^{-1}$ and $\tau_j^{-1}$.

Consider the Butcher tableaus from the Gauß-Legendre and Radau IA method of size $s=2$. The Gauß-Legendre method is given by
\begin{align}\label{eq:gaussleg2}
	\Lambda_{\text{GL}} = \begin{bmatrix}\frac{1}{4} & \frac{1}{4}-\frac{1}{6}\sqrt{3}\\ \frac{1}{4}+\frac{1}{6}\sqrt{3}& \frac{1}{4}\end{bmatrix},\ \beta_{\text{GL}}=\begin{bmatrix}\frac{1}{2}\\ \frac{1}{2}\end{bmatrix}.
\end{align}
This method is equivalent to the Hammer-Hollingsworth method \cite{butcher64} which was used in \cite{Opm12}. The matrix $\Lambda_{\text{GL}}$ has eigenvalues $\mu_{1/2} = \frac{1}{4}\pm\frac{\sqrt{3}}{12}i$. The Radau IA method is given by 
\begin{align}\label{eq:radauia2}
	\Lambda_{\text{R}} = \begin{bmatrix}\frac{1}{4} & -\frac{1}{4}\\ \frac{1}{4}& \frac{5}{12}\end{bmatrix},\ \beta_{\text{R}}=\begin{bmatrix}\frac{1}{4}\\ \frac{3}{4}\end{bmatrix}.
\end{align}
It has eigenvalues $\lambda_{1/2}=\frac{1}{3}\pm\frac{\sqrt{2}}{6}i$.

When \Cref{alg:abt} is executed with the Gauß-Legendre method for $Z_\text{c}$ and the Radau IA method for $Z_\text{o}$, then the moments are matched around the expansion points
\begin{align}
	(\omega_j \mu_{1/2})^{-1} = \omega_j^{-1}(3\mp \sqrt{3}i) \, \text{ and } (\tau_j \lambda_{1/2})^{-1} = \tau_j^{-1}(2\mp \sqrt{2}i)
\end{align}
for $j=1, \dots, N$. These expansion points are visualized in the complex plane in \Cref{fig:exp_points} for $\omega_j=\tau_j=0.3,\, 0.4,\, \dots,\, 1$.

\section{Conclusion}\label{sec:conclusion}
We have presented a method which generates approximate balancing transformations using approximate Cholesky factors of the gramians obtained via numerical quadrature with Runge-Kutta methods. The moments of the reduced system coincide with the moments of the original systems at the inverses of the (conjugated) eigenvalues of the Butcher tableaus multiplied with the time step sizes, while explicit quadrature methods correspond to interpolation at infinity.

It remains an open question how the expansion points can be characterized if the SVD in \Cref{alg:abt} is truncated, i.e. if balanced truncation is performed instead of an approximate balancing transformation. Then the reduced system is obtained via projection onto a subspace of a rational Krylov space and the direct connection between the poles of the rational Krylov space and the expansion points around which the moments are matched is lost.

\bibliographystyle{spmpsci}
\bibliography{literatur}

\end{document}